\documentclass[a4paper]{article}
\usepackage{amsmath}
\usepackage{amssymb}
\usepackage{amsthm}
\usepackage{amsfonts}
\usepackage{microtype}
\usepackage{graphicx}
\usepackage{xspace}
\newtheorem{theorem}{Theorem}

\newtheorem{corollary}[theorem]{Corollary}

\newtheorem{lemma}[theorem]{Lemma}

\newtheorem{proposition}[theorem]{Proposition}

\newcommand{\BDN}{\textbf{BD-}\ensuremath{\mathbf{N}}\xspace}

\begin{document}

\normalfont\sf

\title{Constructive aspects of Riemann's permutation theorem for series}
\author{J. Berger, D. Bridges, H. Diener, H. Schwichtenberg}
\maketitle

\begin{abstract}%
\noindent
The notions of permutable and weak-permutable convergence of a series
$\sum_{n=1}^{\infty}a_{n}$ of real numbers are introduced.
Classically, these two notions are equivalent, and, by Riemann's two main
theorems on the convergence of series, a convergent series is permutably
convergent if and only if it is absolutely convergent. Working within
Bishop-style constructive mathematics, we prove that Ishihara's principle \BDN implies that every permutably convergent series is absolutely convergent. Since there are models of constructive mathematics in which the Riemann permutation theorem for series holds but \BDN does not, the best we can hope for as a
partial converse to our first theorem is that the absolute convergence of
series with a permutability property classically equivalent to that of Riemann
implies \BDN. We show that this is the case when
the property is weak-permutable convergence.

\end{abstract}

\bigskip \bigskip

\section{Introduction}

This paper follows on from \cite{BergB}, in which the first two authors gave
proofs, within the framework of Bishop-style constructive analysis\emph{
\textbf{(BISH)}},\footnote{That is, analysis using intuitionistic logic, a related set theory such as that of Aczel and Rathjen \cite{Aczel}, and dependent choice. For more on \textbf{BISH}, see \cite{Bishop,BB,BV}.} of the two famous series theorems of
Riemann \cite{Riemann}:\footnote{We use shorthand like $\sum a_{n}$ and $\sum
a_{\sigma(n)}$ for series when it is clear what the index of summation is.}

\begin{description}
\item[\textbf{RST}$_{1}$] \emph{If a series }$\sum a_{n}$\emph{\ of real
numbers is absolutely convergent, then for each permutation }$\sigma$ \emph{of
the set }$\mathbf{N}^{+}$ \emph{of positive integers, the series }$\sum
a_{\sigma(n)}$\emph{\ converges to the same sum as }$\sum a_{n}.$

\item[\textbf{RST}$_{2}$] \emph{If a series }$\sum a_{n}$\emph{\ of real
numbers is conditionally convergent, then for each real number }%
$x$\emph{\ there exists a permutation }$\sigma$\emph{\ of }$\mathbf{N}^{+}%
$\emph{\ such that }$\sum a_{\sigma(n)}$\emph{\ converges to }$x$\emph{.}
\end{description}

\noindent
It is not hard to extend the conclusion of \textbf{RST}$_{2}$ to what we call
its \emph{full, extended version}, which includes the existence of
permutations of the series $\sum a_{n}$ that diverge to $\infty$ and to
$-\infty$. In consequence, a simple reductio ad absurdum argument proves
classically\ that if a real series $\sum a_{n}$ is\emph{ \textbf{permutably
convergent}}---that is, every permutation of $\sum a_{n}$ converges in
$\mathbf{R}$---then it is absolutely convergent. An intuitionistic proof of
this last result was provided by Troelstra (\cite{Tro}, pages 95 ff.), using
Brouwer's continuity principle for choice sequences. That result actually has
one serious intuitionistic application: Spitters (\cite{Spitters}, pages
2101--2) uses it to give an intuitionistic proof of the characterisation of
normal linear functionals on the space of bounded operators on a Hilbert
space; he also asks whether there is a proof of the Riemann-Troelstra result
within \textbf{BISH} alone. In Section 3 below, we give a proof, within
\textbf{BISH} \emph{supplemented }by the constructive-foundationally important
principle \BDN, that permutable convergence implies
absolute convergence. While this proof steps outside unadorned \textbf{BISH},
it is valid in both intuitionistic and constructive recursive mathematics, in
which \BDN is derivable.

This raises the question: over \textbf{BISH}, does the absolute convergence of
every permutably convergent series imply \BDN? Thanks to
Diener and Lubarsky \cite{LubDien}, we now know that in certain formal systems
of \textbf{BISH}, the answer is negative; in other words, the result about
permutably convergent series is weaker than \BDN. In turn,
this raises another question: is there a proposition that is
\emph{classically} equivalent to, and clearly cognate with, the absolute
convergence of permutably convergent series and that, added to \textbf{BISH},
implies \BDN? In order to answer this question
affirmatively, we introduce in Section 2 the notion of \emph{weak-permutable
convergence} and then derive some of its fundamental properties, including its
classical equivalence to permutable convergence. In Section 4 we show that the
absolute convergence of weak-permutably convergent series implies
\BDN. Thus, in \textbf{BISH}, we have the implications

\begin{quotation}
 \ \ \ Every weak-permutably convergent series is absolutely convergent%

\bigskip

$\implies$ \BDN

\bigskip

$\implies$ Every permutably convergent series is absolutely convergent.
\end{quotation}

\noindent
In view of the Diener-Lubarsky results in \cite{LubDien}, neither of these
implications can be reversed.

\section{Weak-permutably convergent series in \textbf{BISH}}

By a\emph{ \textbf{bracketing} }of a real series $\sum a_{n}$ we mean a pair comprising

\begin{itemize}
\item a strictly increasing mapping $f:\mathbf{N}^{+}\rightarrow\mathbf{N}%
^{+}$ with $f(1)=1$, and

\item the sequence $\mathbf{b}$ defined by%
\[
b_{k}\equiv\sum_{i=f(k)}^{f(k+1)-1}a_{i}\ \ \ (k\geqslant1).
\]

\end{itemize}

\noindent
We also refer, loosely, to the series $\sum b_{k}$ as a bracketing of $\sum
a_{n}$.

We say that $\sum a_{n}$ is\emph{ \textbf{weak-permutably convergent} }if it
is convergent and if for each permutation $\sigma$ of $\mathbf{N}^{+}$, there
exists a convergent bracketing of $\sum a_{\sigma(n)}$. Clearly, permutable
convergence implies weak-permutable convergence. As we shall see in this
section, the converse holds classically; later we shall show that it does not
hold constructively. As a first step towards this, we have:

\begin{proposition}
\label{060411d}Let $\sum a_{n}$ be a weak-permutably convergent series of real
numbers, with sum $s$, and let $\sigma$ be a permutation of $\mathbf{N}^{+}$.
Then every convergent bracketing of $\sum a_{\sigma(n)}$ converges to $s$.
\end{proposition}

\noindent
The proof of this proposition will depend on some lemmas.

\begin{lemma}
\label{060411a}Let $\sum a_{n}$ be a convergent series of real numbers, with
sum $s$, and let $\sigma$ be a permutation of $\mathbf{N}^{+}$. If there
exists a bracketing $\left(  f,\mathbf{b}\right)  $ of $\sum a_{\sigma(n)}$
that converges to a sum $t\neq s$, then there exist a permutation $\tau$ of
$\mathbf{N}^{+}$ and a strictly increasing sequence $\left(  k_{i}\right)
_{i\geqslant1}$ of positive integers such
\begin{equation}
\left\vert \sum_{n=f(k_{i})+1}^{f(k_{i+1})}a_{\tau(n)}\right\vert >\frac{1}%
{3}\left\vert s-t\right\vert \label{labc0}%
\end{equation}
for all $i$.
\end{lemma}

\begin{proof}
Consider, to illustrate, the case where $s<t$. For convenience, let
$\varepsilon\equiv\frac{1}{3}\left(  t-s\right)  $. Pick $k_{1}$ such that%
\[
\left\vert \sum_{n=j}^{k}a_{n}\right\vert <\frac{\varepsilon}{2}\text{
\ \ }\left(  k>j>f(k_{1})\right)  .
\]
Then $\sum_{n=1}^{f(k_{1})}a_{n}<s+\varepsilon$. Set $\tau(k)\equiv k$ for
$1\leqslant k\leqslant f(k_{1})$. Next pick $k_{2}>k_{1}$ such that

\begin{itemize}
\item $\left\{  1,\ldots,f(k_{1})\right\}  \subset\left\{  \sigma
(n):1\leqslant n\leqslant f(k_{2})\right\}  $ and

\item $\left\vert \sum_{n=f(j)}^{f(k)}a_{\sigma(n)}\right\vert <\varepsilon/2$
whenever $k>j>f(k_{2})$.
\end{itemize}

\noindent
Define $\tau(n)$ for $f(k_{1})<n\leqslant f(k_{2})$ so that%
\[
\left\{  \sigma(n)\mid1\leqslant n\leqslant f(k_{2}),\ \sigma(n)>f(k_{1}%
)\right\}  =\left\{  \tau(f(k_{1})+1),\ldots,\tau(f(k_{2}))\right\}
\]
Note that%
\[
\sum_{n=1}^{f(k_{2})}a_{\tau(n)}=\sum_{n=1}^{f(k_{2})}a_{\sigma(n)}%
>t-\varepsilon.
\]
Next, pick $k_{3}>k_{2}$ such that%
\[
\left\{  \tau(1),\dots,\tau(f(k_{2}))\right\}  \subset\left\{  1,\dots
,f(k_{3})\right\}
\]
Define $\tau(n)$ for $f(k_{2})<n\leqslant f(k_{3})$ so that%
\[
\left\{  n:1\leqslant n\leqslant f(k_{3}),~n>\tau\left(  f(k_{2})\right)
\right\}  =\left\{  \tau(f(k_{2})+1),\dots,\tau(f(k_{3}))\right\}  .
\]
Then%
\[
\sum_{n=1}^{f(k_{3})}a_{\tau(n)}=\sum_{n=1}^{f(k_{3})}a_{n}<s+\varepsilon.
\]
Carrying on in this way, we construct, inductively, a strictly increasing
sequence $\left(  k_{i}\right)  _{i\geqslant1}$ of positive integers, and a
permutation $\tau$ of $\mathbf{N}^{+}$, such that for each $j$,%
\[
\sum_{n=1}^{f(k_{2j-1})}a_{\tau(n)}<s+\varepsilon\text{ \ and \ }\sum
_{n=1}^{f(k_{2j})}a_{\tau(n)}>t-\varepsilon.
\]
When $i\in\mathbf{N}^{+}$ is even, we obtain
\[
\left\vert \sum_{n=f(k_{i})+1}^{f(k_{i+1})}a_{\tau(n)}\right\vert
\geqslant\sum_{n=1}^{f(k_{i})}a_{\tau(n)}-\sum_{n=1}^{f(k_{i+1})}a_{\tau
(n)}>t-s-2\varepsilon>\frac{1}{3}\left(  t-s\right)  .
\]
A similar argument gives (\ref{labc0}) when $i$ is odd.
\end{proof}

\begin{lemma}
\label{060411b}Under the hypotheses of \emph{Lemma \ref{060411a}}, the series
$\sum\left\vert a_{n}\right\vert $ diverges.
\end{lemma}

\begin{proof}
Construct the permutation $\tau$ and the sequence $\left(  k_{i}\right)
_{i\geqslant1}$ as in Lemma \ref{060411a}. Given $C>0$, compute $j$ such that
$\left(  j-1\right)  \left\vert s-t\right\vert >3C$. Then
\[
\sum_{n=1}^{f(k_{j})}\left\vert a_{\tau(n)}\right\vert \geqslant\sum
_{i=1}^{j-1}\left\vert \sum_{n=f(k_{i})+1}^{f(k_{i+1})}a_{\tau(n)}\right\vert
>\frac{j-1}{3}\left\vert s-t\right\vert >C.
\]
Then compute $M$ such that%
\[
\left\{  a_{\tau(1)},\ldots,a_{\tau(f(k_{j}))}\right\}  \subset\left\{
a_{1},\ldots,a_{M}\right\}  .
\]
Then%
\[
\sum_{n=1}^{M}\left\vert a_{n}\right\vert \geqslant\sum_{n=1}^{f(k_{j}%
)}\left\vert a_{\tau(n)}\right\vert >C.
\]
Since $C>0$ is arbitrary, the conclusion follows.
\end{proof}

\begin{lemma}
\label{060411c}Let $\sum a_{n}$ be a convergent series of real numbers, and
$\tau$ a permutation of $\mathbf{N}^{+}$ such that $\sum a_{\tau(n)}$ diverges
to infinity. Then it is impossible that $\sum a_{\tau(n)}$ have a convergent bracketing.
\end{lemma}

\begin{proof}
Suppose there exists a bracketing $\left(  f,\mathbf{b}\right)  $ of $\sum
a_{\tau(n)}$ that converges to a sum $s$. Compute $N>1$ such that%
\begin{equation}
\sum_{n=1}^{\nu}a_{\tau(n)}>s+1\ \ \ \left(  \nu\geqslant N\right)  \text{.}
\label{labc1}%
\end{equation}
There exists $N_{1}>N$ such that
\[
\left\vert \sum_{i=1}^{N_{1}}\sum_{n=f(i)}^{f(i+1)-1}a_{\tau(n)}-s\right\vert
<1
\]
and therefore%
\[
\left\vert \sum_{n=1}^{f(N_{1}+1)-1}a_{\tau(n)}\right\vert <s+1.
\]
Since $f(N_{1}+1)>N$, this contradicts (\ref{labc1}).
\end{proof}

\begin{lemma}
\label{101111a}Let $\sum a_{n}$ be a weak-permutably convergent series of real
numbers, and $\sigma$ a permutation of $\mathbf{N}^{+}$. Then it is impossible
that $\sum\left\vert a_{\sigma(n)}\right\vert $ diverge.
\end{lemma}

\begin{proof}
Suppose that $\sum\left\vert a_{\sigma(n)}\right\vert $ does diverge. Then, by
the full, extended version of \textbf{RST}$_{2}$, there is a permutation
$\tau$ of $\mathbf{N}^{+}$ such that $\sum a_{\tau(n)}$ diverges to infinity.
Since $\sum a_{n}$ is weak-permutably convergent, there exists a bracketing of
$\sum a_{\tau(n)}$ that converges. This is impossible, in view of Lemma
\ref{060411c}.
\end{proof}

\bigskip

Arguing with classical logic, we see that if $\sum a_{n}$ is weak-permutably
convergent, then, by Lemma \ref{101111a}, $\sum\left\vert a_{n}\right\vert $
must converge; whence $\sum a_{n}$ is permutably convergent, by \textbf{RST}%
$_{1}$.

Returning to intuitionistic logic, we have reached the \textbf{proof of
Proposition \ref{060411d}}:%

\bigskip

\begin{proof}
Suppose that there exists a bracketing of $\sum a_{\sigma(n)}$ that converges
to a sum distinct from $s$. Then, by Lemma \ref{060411b}, $\sum\left\vert
a_{n}\right\vert $ diverges. Lemma \ref{101111a} shows that this is
impossible. It follows from the tightness of the inequality on $\mathbf{R}$
that every convergent bracketing of $\sum a_{\sigma(n)}$ converges to $s$.
\end{proof}

\bigskip

Since permutable convergence implies convergence and is a special case of
weak-permutable convergence, we also have:

\begin{corollary}
\label{070411a}Let $\sum a_{n}$ be a permutably convergent series of real
numbers, and let $\sigma$ be a permutation of $\mathbf{N}$. Then $\sum
a_{\sigma(n)}=\sum a_{n}$. \label{ere 111111}
\end{corollary}

\section{\BDN and permutable convergence}

A subset $S$ of $\mathbf{N}^{+}$ is said to be\emph{ \textbf{pseudobounded}
}if for each sequence $\left(  s_{n}\right)  _{n\geqslant1}$ in $S$, there
exists $N$ such that $s_{n}/n<1$ for all $n\geqslant N$---or, equivalently, if
$s_{n}/n\rightarrow0$ as $n\rightarrow\infty$. Every bounded subset of
$\mathbf{N}^{+}$ is pseudobounded; the converse holds classically,
intuitionistically, and in recursive constructive mathematics, but Lietz
\cite{Lietz} and Lubarsky \cite{Lub} have produced models of \textbf{BISH} in
which it fails to hold for inhabited, countable, pseudobounded sets. Thus the principle

\begin{description}
\item[\BDN] \emph{Every inhabited, countable,
pseudobounded subset of} $\mathbf{N}^{+}$ \emph{is bounded}%
\footnote{\BDN was introduced by Ishihara in
\cite{IshJSL572} (see also \cite{Fred}).}
\end{description}

\noindent
is independent of \textbf{BISH. }It is a serious problem of constructive
reverse mathematics \cite{cerisy,Ishrev,Ishrev2} to determine which classical
theorems are equivalent to \textbf{BISH} + \BDN. For
example, it is known that the full form of Banach's inverse mapping theorem in
functional analysis is equivalent, over \textbf{BISH}, to \textbf{BD}%
-$\mathbf{N}$; see \cite{Ishunknown}.

This section is devoted to our version of the\emph{ \textbf{Riemann
permutability theorem}: }

\begin{theorem}
\label{2104a1}In\emph{ \textbf{BISH} + \textbf{BD}-}$\mathbf{N}$, every
permutably convergent series of real numbers is absolutely convergent.
\end{theorem}

\begin{proof}
Let $\sum_{i=1}^{\infty}a_{i}$ be a permutably convergent series of real
numbers. To begin with, assume that each $a_{i}$ is rational. Write%
\[
a_{n}^{+}=\max\left\{  a_{n},0\right\}  ,\ a_{n}^{-}=\max\left\{
-a_{n},0\right\}  .
\]
Given a positive rational number $\varepsilon$, define a binary mapping $\phi$
on $\mathbf{N}^{+}\times\mathbf{N}^{+}$\ such that%
\begin{align*}
\phi\left(  m,n\right)  =0  &  \Rightarrow m>n\wedge\sum_{i=n+1}^{m}a_{i}%
^{+}\geqslant\varepsilon,\\
\phi(m,n)=1  &  \Rightarrow m\leqslant n\vee\sum_{i=n+1}^{m}a_{i}%
^{+}<\varepsilon.
\end{align*}
We may assume that $\phi\left(  2,1\right)  =0$. Let%
\[
S\equiv\left\{  n:\exists_{m}\left(  \phi(m,n)=0\right)  \right\}  .
\]
Then $S$ is countable and downward closed. In order to prove that $S$ is
pseudobounded, let $\left(  s_{n}\right)  _{n\geqslant1}$ be an increasing
sequence in $S$. We may assume that $s_{1}=1$. Define a map $\kappa
:S\rightarrow\mathbf{N}^{+}$ by
\[
\kappa(n)\equiv\min\left\{  m:m>n\wedge\sum_{i=n+1}^{m}a_{i}^{+}%
\geqslant\varepsilon\right\}  .
\]
Setting $\lambda_{1}=0$, we construct inductively a binary sequence
$\mathbf{\lambda}\equiv\left(  \lambda_{n}\right)  _{n\geqslant1}$ with the
following properties:
\begin{equation}
\forall_{n}\left(  \left(  \lambda_{n}=0\wedge\lambda_{n+1}=1\right)
\Rightarrow n+1\in S\right)  \label{sigma0}%
\end{equation}%
\begin{equation}
\forall_{n}\,\exists_{m}\left(  \lambda_{n}=1\Rightarrow\lambda_{n+m}%
=0\right)  \label{nopanic}%
\end{equation}%
\begin{equation}
\forall_{n}\left(  \left(  \lambda_{n}=0\wedge\lambda_{n+1}=0\right)
\Rightarrow s_{n+1}\leqslant n+1\right)  \label{event}%
\end{equation}
Suppose that $\lambda_{1},\dots,\lambda_{n}$ have been defined such that
\begin{equation}
\forall_{k<n}\left(  \left(  \lambda_{k}=0\wedge\lambda_{k+1}=1\right)
\Rightarrow k+1\in S\right)  . \label{aux}%
\end{equation}
\smallskip In the case $\lambda_{n}=0$, if $s_{n+1}\leqslant n+1$, we set
$\lambda_{n+1}=0$; and if $s_{n+1}>n+1$, we set $\lambda_{n+1}=1$, noting that
$n+1\in S$ since $S$ is downward closed. In the case $\lambda_{n}=1$, we
define
\[
n^{\prime}\equiv\min\left\{  i\leqslant n:\forall_{j}\left(  i\leqslant
j\leqslant n\Rightarrow\lambda_{j}=1\right)  \right\}  .
\]
Then the hypothesis (\ref{aux}) ensures that $n^{\prime}\in S$. If
$\kappa(n^{\prime})=n$, then $\sum_{i=n^{\prime}+1}^{n}a_{i}^{+}%
\geqslant\varepsilon$ and we set $\lambda_{n+1}=0$; otherwise, we set
$\lambda_{n+1}=1$. This concludes the inductive construction of the sequence
$\mathbf{\lambda}$. Note that in the case $\lambda_{n}=\lambda_{n+1}=1$, this
construction will eventually give $\lambda_{n+1+m}=0$ for some $m$, since%
\[
\kappa(n^{\prime})\geqslant n+1,\sum_{i=n^{\prime}+1}^{\kappa(n^{\prime}%
)-1}a_{i}^{+}<\varepsilon\text{, and }\sum_{i=n^{\prime}+1}^{\kappa(n^{\prime
})}a_{i}^{+}\geqslant\varepsilon.
\]
Hence the sequence $\mathbf{\lambda}$ has all three properties (\ref{sigma0}%
)--(\ref{event}).

\smallskip For convenience, if $n\leqslant m$ and the following hold, we call
the interval $I=\left[  n,m\right]  $ of $\mathbf{N}^{+}$ a \emph{bad
interval}:

\begin{itemize}
\item[--] if $n>1$ then $\lambda_{n-1}=0$,

\item[--] $\lambda_{m+1}=0$, and

\item[--] $\lambda_{i}=1$ for all $i\in I$.
\end{itemize}

\noindent
Define a permutation $\sigma$ of $\mathbf{N}^{+}$ as follows. If $\lambda
_{n}=0$, then $\sigma(n)\equiv n$. If $\left[  n,m\right]  $ is a bad
interval, then the construction of the sequence $\mathbf{\lambda}$ ensures
that $\kappa(n)=m$, so $\sum_{i=n+1}^{m}a_{i}^{+}\geqslant\varepsilon$. Let
$\sigma$ map an initial segment $\left[  n,n+k-1\right]  $ of $\left[
n,m\right]  $ onto%
\[
\left\{  i:n\leqslant i\leqslant m\wedge a_{i}^{+}>0\right\}  ,
\]
and map the remaining elements of $\left[  n,m\right]  $ onto%
\[
\left\{  i:n\leqslant i\leqslant m\wedge a_{i}^{+}=0\right\}  .
\]
Note that for all $n\geqslant1$,%
\begin{equation}
\left(  \lambda_{n-1}=0\wedge\lambda_{n}=1\right)  \Rightarrow\exists
_{j,k}\left(  n\leqslant j<k\wedge\sum_{i=j+1}^{k}a_{\sigma(i)}\geqslant
\varepsilon\right)  . \label{sig1}%
\end{equation}
Since $\sum_{i=1}^{\infty}a_{\sigma(i)}$ is convergent, there exists $J$ such
that $\sum_{i=j+1}^{k}a_{\sigma(i)}<\varepsilon$ whenever $J\leqslant j<k$. In
view of (\ref{nopanic}), we can assume that $\lambda_{J}=0$. If $n\geqslant J$
and $\lambda_{J}=1$, then there exists $\nu$ such that $J\leqslant\nu<n$,
$\lambda_{\nu}=0$, and $\lambda_{\nu+1}=1$; whence there exist $j,k$ such that
$J\leqslant\nu\leqslant j<k$ and $\sum_{i=j+1}^{k}a_{\sigma(i)}\geqslant
\varepsilon$, a contradiction. Thus $\lambda_{n}=0$ for all $n\geqslant J$,
and therefore, by (\ref{event}), $s_{n}\leqslant n$ for all $n>J$. This
concludes the proof that $S$ is pseudobounded.

Applying \BDN, we obtain a positive integer $N$ such that
$n<N$ for all $n\in S$. If $m>n\geqslant N$ and $\sum_{i=n+1}^{m}a_{i}%
^{+}>\varepsilon$, then $\phi\left(  m,n\right)  \neq1$, so $\phi(m,n)=0$ and
therefore $n\in S$, a contradiction. Hence $\sum_{i=n+1}^{m}a_{i}^{+}%
\leqslant\varepsilon$ whenever $m>n\geqslant N$. Likewise, there exists
$N^{\prime}$ such that $\sum_{i=n+1}^{m}a_{i}^{-}\leqslant\varepsilon$
whenever $m>n\geqslant N^{\prime}$. Thus if $m>n\geqslant\max\left\{
N,N^{\prime}\right\}  $, then%
\[
\sum_{i=n+1}^{m}\left\vert a_{i}\right\vert =\sum_{i=n+1}^{m}a_{i}^{+}%
+\sum_{i=n+1}^{m}a_{i}^{-}\leqslant2\varepsilon.
\]
Since $\varepsilon>0$ is arbitrary, we conclude that the partial sums of the
series $\sum\left\vert a_{n}\right\vert $ form a Cauchy sequence, and hence
that the series converges.

It remains to remove the restriction that the terms $a_{i}$ be rational. In
the general case, for each $i$ pick $b_{i}$ such that $a_{i}+b_{i}$ is
rational and $0<b_{i}<2^{-i}$. Note that the series $\sum_{i=1}^{\infty}b_{i}$
converges absolutely and so, by \textbf{RST}$_{1}$, is permutably convergent.
Hence $\sum_{i=1}^{\infty}(a_{i}+b_{i})$ is permutably convergent. By the
first part of the proof, $\sum_{i=1}^{\infty}\left\vert a_{i}+b_{i}\right\vert
$ is convergent, as therefore is $\sum_{i=1}^{\infty}\left\vert a_{i}%
\right\vert $.
\end{proof}

\section{Weak-permutable convergence and \BDN}

Diener and Lubarsky \cite{LubDien} have recently constructed topological
models showing that the absolute convergence of every permutably convergent
series in $\mathbf{R}$ neither implies \BDN nor is
provable within the Aczel-Rathjen set-theoretic formulation of \textbf{BISH
}\cite{Aczel}, and may therefore be of constructive reverse-mathematical
significance in its own right. Their models lead us to ask: is there a variant
of the Riemann permutability theorem that is \emph{classically equivalent} to
the original form and that implies \BDN? Since
weak-permutable and permutable convergence are classically equivalent, the
main result of this section provides an affirmative answer:

\begin{theorem}
\label{060411f}The statement

\begin{description}
\item \emph{(*) \ \ }Every weak-permutably convergent series in $\mathbf{R}$
is absolutely convergent
\end{description}

\noindent
implies \BDN.
\end{theorem}

The hard part of the proof is isolated in the complicated construction in the
following lemma.

\begin{lemma}
\label{2309a1}Let $S\equiv\left\{  s_{1},s_{2},\ldots\right\}  $ be an
inhabited, countable, pseudobounded subset of $\mathbf{N}$. Then there exists
a sequence $\left(  a_{n}\right)  _{n\geqslant1}$ of nonnegative rational
numbers with the following properties.

\begin{itemize}
\item[\emph{(i)}] $\sum\left(  -1\right)  ^{n}a_{n}$ is convergent and
weak-permutably convergent.

\item[\emph{(ii)}] If $\sum a_{n}$ converges, then $S$ is bounded.
\end{itemize}
\end{lemma}

\begin{proof}%
\noindent
To perform this construction, we first replace each $s_{n}$ by $\max\left\{
s_{k}:k\leqslant n\right\}  $, thereby obtaining $s_{1}\leqslant
s_{2}\leqslant\cdots$. Now construct a binary sequence $\left(  \lambda
_{k}\right)  _{k\geqslant1}$ such that%
\begin{align*}
\lambda_{k}=0  &  \Rightarrow s_{2^{k+1}}=s_{2^{k}},\\
\lambda_{k}=1  &  \Rightarrow s_{2^{k+1}}>s_{2^{k}}.
\end{align*}
For $2^{k}+1\leqslant n+1<2^{k+1}$, set $a_{n}=\lambda_{k}/\left(  n+1\right)
$. Note that if $\lambda_{k}=1$, then $\sum_{n=2^{k}+1}^{2^{k+1}}a_{n}%
>\frac{1}{2}$. In order to show that $\sum_{n=1}^{\infty}\left(  -1\right)
^{n}a_{n}$ converges in $\mathbf{R}$, first observe that if $\lambda_{k}=1$
and $2^{k}<m_{1}\leqslant m_{2}\leqslant2^{k+1}$, then%
\begin{equation}
\left\vert \sum_{n=m_{1}}^{m_{2}}\left(  -1\right)  ^{n}a_{n}\right\vert
=\left\vert \sum_{n=m_{1}}^{m_{2}}\frac{\left(  -1\right)  ^{n}}%
{n+1}\right\vert <\frac{1}{2^{k}}. \label{0903a}%
\end{equation}
If $j,k,m_{1},m_{2}$ are positive integers with $2^{k}<m_{1}\leqslant
2^{k+1}\leqslant2^{j}<m_{2}\leqslant2^{j+1}$, then%
\begin{align*}
&  \left\vert \sum_{n=m_{1}}^{m_{2}}\left(  -1\right)  ^{n}a_{n}\right\vert \\
&  \leqslant\left\vert \sum_{n=m_{1}}^{2^{k+1}}\left(  -1\right)  ^{n}%
a_{n}\right\vert +\sum_{\substack{k<\nu<j,\\\lambda_{\nu}=1}}\left\vert
\sum_{n=2^{\nu}+1}^{2^{\nu+1}}\left(  -1\right)  ^{n}a_{n}\right\vert
+\left\vert \sum_{n=2^{j}+1}^{m_{2}}\left(  -1\right)  ^{n}a_{n}\right\vert \\
&  \leqslant\frac{1}{2^{k}}+\sum_{\substack{k<\nu<j,\\\lambda_{\nu}=1}%
}\frac{1}{2^{\nu}}+\frac{1}{2^{j}}\\
&  \leqslant\sum_{n=k}^{\infty}\frac{1}{2^{n}}=\frac{1}{2^{k-1}}.
\end{align*}
Hence the partial sums of $\sum_{n=1}^{\infty}\left(  -1\right)  ^{n}a_{n}$
form a Cauchy sequence, and so the series converges to a sum $s\in\mathbf{R}$.

Consider any permutation $\sigma$ of $\mathbf{N}^{+}$. In order to show that
$\sum_{n=1}^{\infty}\left(  -1\right)  ^{\sigma(n)}a_{\sigma(n)}$ converges,
we construct strictly increasing sequences $\left(  j_{k}\right)
_{k\geqslant1}$ and $\left(  n_{k}\right)  _{k\geqslant1}$ of positive
integers such that for each $k$,

\begin{itemize}
\item[(a)] $2^{j_{k}}<n_{k}<2^{j_{k+1}}$,

\item[(b)] $\left\{  n:n+1<2^{j_{k}}\right\}  \subset\left\{  \sigma
(n):n+1<n_{k}\right\}  \subset\left\{  1,2,\ldots2^{j_{k+1}}-1\right\}  $, and

\item[(c)] $\left\vert \sum_{n=2^{j_{k}}}^{i}\left(  -1\right)  ^{n}%
a_{n}\right\vert <2^{-k+1}$ for all $k\geqslant1$ and $i\geqslant2^{j_{k}}$.
\end{itemize}

\noindent
Setting $j_{1}=2$, pick $n_{1}>4$ such that%
\[
\left\{  1,2\right\}  \subset\left\{  \sigma(n):n+1<n_{1}\right\}  .
\]
Then pick $j_{2}>j_{1}$ such that $n_{1}<2^{j_{2}}$,%
\[
\left\{  \sigma(n):n+1<n_{1}\right\}  \subset\left\{  n:n+1<2^{j_{2}}\right\}
,
\]
and $\left\vert \sum_{n=2^{j_{2}}}^{i}\left(  -1\right)  ^{n}a_{n}\right\vert
\,<2^{-1}$ for all $i\geqslant2^{j_{2}}$. Next pick, in turn, $n_{2}>2^{j_{2}%
}$ and $j_{3}>j_{2}$ such that%
\[
\left\{  n:n+1<2^{j_{2}}\right\}  \subset\left\{  \sigma(n):n+1<n_{2}\right\}
\subset\left\{  n:n+1<2^{j_{3}}\right\}
\]
and $\left\vert \sum_{n=2^{j_{3}}}^{i}\left(  -1\right)  ^{n}a_{n}\right\vert
<2^{-2}$ for all $i\geqslant2^{j_{3}}$. Carrying on in this way, we complete
the construction of our sequences $\left(  j_{k}\right)  _{k\geqslant
1},\left(  n_{k}\right)  _{k\geqslant1}$ with properties (a)--(c).

Now consider the sequence $\left(  s_{2^{j_{k}+1}}\right)  _{k\geqslant1}$.
Since $S$ is pseudobounded, there exists a positive integer $K_{1}$ such that
$s_{2^{j_{k+1}}}<k$ for all $k\geqslant K_{1}$. Suppose that for each positive
integer $k\leqslant K_{1}$, there exists $i_{k}$ such that $j_{k}\leqslant
i_{k}<j_{k+1}$ and $\lambda_{i_{k}}=1$. Then%
\[
s_{2^{i_{1}}}<s_{2^{i_{2}}}<\cdots<s_{2^{i_{K_{1}}}}<s_{2^{j_{K_{1}+1}}},
\]
so $K_{1}\leqslant s_{2^{j_{K_{1}+1}}}<K_{1}$, a contradiction. Hence there
exists $k_{1}\leqslant K_{1}$ such that for each $i$ with $j_{k_{1}}\leqslant
i<j_{k_{1}+1}$, we have $\lambda_{i}=0,$ and therefore $a_{n}=0$ whenever
$2^{i}\leqslant n+1<2^{i+1}$. Thus $a_{n}=0$ whenever $2^{j_{k_{1}}}\leqslant
n+1<2^{j_{k_{1}+1}}$. It follows that%
\begin{align*}
\left\{  a_{n}:n+1<2^{j_{k_{1}}}\right\}   &  \subset\left\{  a_{\sigma
(n)}:n+1<n_{k_{1}}\right\} \\
&  \subset\left\{  a_{n}:n+1<2^{j_{k_{1}+1}}\right\} \\
&  =\left\{  a_{n}:n+1<2^{j_{k_{1}}}\right\}  \cup\left\{  a_{n}:2^{j_{k_{1}}%
}\leqslant n+1<2^{j_{k_{1}+1}}\right\} \\
&  =\left\{  a_{n}:n+1<2^{j_{k_{1}}}\right\}  \cup\left\{  0\right\}  .
\end{align*}
Without loss of generality, we may assume that $a_{1}=0$. Then%
\[
\left\{  a_{n}:n+1<2^{j_{k_{1}}}\right\}  =\left\{  a_{\sigma(n)}%
:n+1<n_{k_{1}}\right\}  .
\]

Next consider the sequence $\left(  s_{2^{j_{k_{1}+k+1}}}\right)
_{k\geqslant1}$. Since $S$ is pseudobounded, there exists a positive integer
$K_{2}$ such that $s_{2^{j_{k_{1}+k+1}}}<k$ for all $k\geqslant K_{2}$.
Suppose that for each positive integer $k\leqslant K_{2}$, there exists
$i_{k}$ such that $j_{k_{1}+k}\leqslant i_{k}<j_{k_{1}+k+1}$ and
$\lambda_{i_{k}}=1$. Then%
\[
s_{2^{i_{1}}}<s_{2^{i_{2}}}<\cdots<s_{2^{i_{K_{2}}}}<s_{2^{j_{k_{1}+K_{2}+1}}%
},
\]
so $K_{2}\leqslant s_{2^{j_{k_{1}+K_{2}+1}}}<K_{2}$, which is absurd. Hence
there exists $\kappa\leqslant K_{2}$ such that for each $i$ with
$j_{k_{1}+\kappa}\leqslant i<j_{k_{1}+\kappa+1}$, we have $\lambda_{i}=0$, and
therefore $a_{n}=0$ whenever $2^{i}\leqslant n+1<2^{i+1}$. Setting
$k_{2}\equiv k_{1}+\kappa$, we have $a_{n}=0$ for all $n$ with $2^{j_{k_{2}}%
}\leqslant n+1<2^{j_{k_{2}+1}}$. Hence%
\begin{align*}
\left\{  a_{n}:n+1<2^{j_{k_{2}}}\right\}   &  \subset\left\{  a_{\sigma
(n)}:n+1<n_{k_{2}}\right\} \\
&  \subset\left\{  a_{n}:n+1<2^{j_{k_{2}+1}}\right\} \\
&  =\left\{  a_{n}:n+1<2^{j_{k_{2}}}\right\}  \cup\left\{  a_{n}:2^{j_{k_{2}}%
}\leqslant n+1<2^{j_{k_{2}+1}}\right\} \\
&  =\left\{  a_{n}:n+1<2^{j_{k_{2}}}\right\}  \cup\left\{  0\right\}  .
\end{align*}
Thus, since $a_{1}=0$,%
\[
\left\{  a_{n}:n+1<2^{j_{k_{2}}}\right\}  =\left\{  a_{\sigma(n)}%
:n+1<n_{k_{2}}\right\}  .
\]
Carrying on in this way, we construct positive integers $k_{1}<k_{2}%
<k_{3}<\cdots$ such that for each $i$,%
\[
\left\{  a_{n}:n+1<2^{j_{k_{i}}}\right\}  =\left\{  a_{\sigma(n)}%
:n+1<n_{k_{i}}\right\}  .
\]
Since both $\sigma$ and $\sigma^{-1}$ are injective, it readily follows that
for each $i$,%
\[
\left\{  \sigma(n):n_{k_{i}}\leqslant n+1<n_{k_{i+1}}\right\}  =\left\{
m:2^{j_{k_{i}}}\leqslant m<2^{j_{k_{i+1}}}\right\}
\]
and therefore%
\[
\left\vert \sum_{n=n_{k_{i}}}^{n_{k_{i+1}}-1}\left(  -1\right)  ^{\sigma
(n)}a_{\sigma(n)}\right\vert =\left\vert \sum_{m=2^{j_{k_{i}}}}^{2^{j_{k_{i}%
+1}-1}}\left(  -1\right)  ^{m}a_{m}\right\vert <\frac{1}{2^{k_{i}}}.
\]
We now see that%
\[
\sum_{i=1}^{\infty}\sum_{n=n_{k_{i}}}^{n_{k_{i+1}-1}}(-1)^{\sigma(n)}%
a_{\sigma(n)}%
\]
converges, by comparison with $\sum_{i=1}^{\infty}2^{-k_{i}}$. Thus
$\sum_{n=1}^{\infty}a_{n}$ is weak-permutably convergent.

Finally, suppose that $\sum_{n=1}^{\infty}a_{n}$ converges. Then there exists
$N$ such that $\sum_{n=N+1}^{\infty}a_{n}<1/2$. It follows that $\lambda
_{n}=0$, and therefore that $s_{n}=s_{2^{N}}$, for all $n\geqslant N$; whence
$s_{n}\leqslant s_{2^{N}}$ for all $n$, and therefore $S$ is a bounded set.
\end{proof}

\bigskip

The proof of Theorem \ref{060411f} is now straightforward:%

\bigskip

\begin{proof}
Given an inhabited, countable, pseudobounded subset $S$ of $\mathbf{N}$,
construct a sequence $\left(  a_{n}\right)  _{n\geqslant1}$ of nonnegative
rational numbers with properties (i) and (ii) in Lemma \ref{2309a1}. Assuming
(*), we see that $\sum a_{n}$ converges; whence, by property (ii), $S$ is a
bounded set.
\end{proof}

\section{Concluding remarks}

We have shown that, over \textbf{BISH},

\begin{itemize}
\item[--] with \BDN, every permutably convergent series is
absolutely convergent;

\item[--] the absolute convergence of every weak-permutably convergent series
implies \BDN.
\end{itemize}

\noindent
It follows from the latter result that if weak-permutable convergence
constructively implies, and is therefore equivalent to, permutable
convergence, then the absolute convergence of every permutably convergent
series implies, and is therefore equivalent to, \BDN.
Since the topological models in \cite{LubDien} show that this is not the case,
we see that, relative to \textbf{BISH}, weak-permutable convergence is a
strictly weaker notion than permutable convergence. In fact, the
Diener-Lubarsky result shows that there is no algorithm which, applied to any
inhabited, countable, pseudobounded subset $S$ of $\mathbf{N}$ and the
corresponding weak-permutably convergent series $\sum a_{n}$ constructed in
the proof of Lemma \ref{2309a1}, proves that that series is permutably
convergent. Nevertheless, weak-permutable convergence and permutable
convergence are classically equivalent notions; the constructive distinction
between them is that the former implies, but is not implied by, \textbf{BD}%
-$\mathbf{N}$, which in turn implies, but is not implied by, the latter.%

\bigskip

\bigskip

\noindent
\textbf{Acknowledgements.} This work was supported by (i) a Marie Curie IRSES
award from the European Union, with counterpart funding from the Ministry of
Research, Science \& Technology of New Zealand, for the project
\emph{Construmath}; and (ii) a Feodor Lynen Return Fellowship for Berger, from
the Humboldt Foundation. The authors also thank the Department of Mathematics
\& Statistics at the University of Canterbury, for releasing Bridges to visit
Munich under the terms of the IRSES award.%

\bigskip

\bigskip

\bigskip

\noindent
\textbf{Keywords:} \ Permutation of series, constructive reverse mathematics%

\noindent
\textbf{MR Classifications (2010): }03F60, 26A03, 26E40%

\bigskip

\noindent
\textbf{Authors' addresses:}

\begin{description}
\item \emph{Berger:} Institut f\"{u}r Mathematik und Informatik,
Walther-Rathenau-Stra\ss e 47, D-17487 Greifswald, Germany

\begin{description}
\item {\small \texttt{jberger@math.lmu.de}}
\end{description}

\item \emph{Bridges:} Department of Mathematics \& Statistics, University of
Canterbury, Private Bag 4800, Christchurch 8140, New Zealand

\begin{description}
\item {\small \texttt{d.bridges@math.canterbury.ac.nz}}
\end{description}

\item \emph{Diener:} Fakult\"{a}t IV: Mathematik, Emmy-Noether-Campus, Walter-Flex-Str. 3, 57072 Siegen, Germany

\begin{description}
\item {\small \texttt{diener@math.uni-siegen.de}}
\end{description}

\item \emph{Schwichtenberg:} Mathematisches Institut der LMU, Theresienstr.
39, 80333 M\"{u}nchen, Germany
\begin{description}
\item {\small \texttt{schwicht@math.lmu.de}}{}
\end{description}
\end{description}

\end{document}